\def\ArxivFormat{1}

\if\ArxivFormat1
  \documentclass[10pt]{amsart}
\else 
\fi

\if\ArxivFormat1
  \usepackage{cmap}
  \usepackage[T1]{fontenc}
  \usepackage[utf8x]{inputenc}
  \usepackage{tabularx}
  \usepackage{amsmath,amssymb,amsfonts}
  
  \usepackage{cite}
  \usepackage[pdfstartview=FitH, bookmarks, pdfpagemode=UseOutlines,%
  hidelinks, plainpages=false, pdfpagelabels, pdfpagelayout=OneColumn,%
  naturalnames=true, hypertexnames=false]{hyperref}
  \usepackage{nicefrac}
  \usepackage[inline]{enumitem}
  \usepackage{amsthm}  
  \usepackage{amsaddr}
  
  \newtheorem{theorem}{Theorem}[section]
  \newtheorem{corollary}{Corollary}[theorem]
  \newtheorem{lemma}[theorem]{Lemma}

  \newtheorem{definition}[theorem]{Definition}
  \newtheorem{example}[theorem]{Example}
  \newtheorem{xca}[theorem]{Exercise}

  \theoremstyle{remark}
  \newtheorem{remark}[theorem]{Remark}

  \numberwithin{equation}{section}
  
  \usepackage[%
  a4paper,%
  twoside=false,%
  layoutvoffset=5mm,%
  textheight=227mm,
  textwidth=142mm,
  ignorehead,%
  ignorefoot,%
  ignoremp,%
  hcentering]{geometry}

\else 
  \usepackage[T1]{fontenc}
  \usepackage[utf8]{inputenc}
  \usepackage{amsmath}

  \usepackage{amsthm}  
  \usepackage[version=4]{mhchem}
  \usepackage{tabularx,longtable}
  \newtheorem{theorem}{Theorem}[section]
  
  \newtheorem{lemma}[theorem]{Lemma}

  \newtheorem{definition}[theorem]{Definition}

  \theoremstyle{remark}
  \newtheorem{remark}[theorem]{Remark}

  \numberwithin{equation}{section}

\fi

\usepackage{xspace}
\usepackage{graphicx}
\usepackage{textcomp}
\usepackage{xcolor}
\usepackage{flushend}
\usepackage{mathtools}
\usepackage{soul}
\usepackage{etoolbox}
\usepackage{twoopt}
\usepackage[ruled, linesnumbered, vlined]{algorithm2e}
\usepackage{algpseudocode}
\usepackage{subcaption}

\usepackage{booktabs}   
\usepackage[table]{xcolor} 
\usepackage{siunitx}    
\usepackage{array}

\newcolumntype{L}{>{\raggedright\arraybackslash}X}

\setlist%
{%
 topsep=0pt,%
 labelsep=6pt,
 noitemsep,%
 leftmargin=*
}

\newcolumntype{Y}{>{\centering\arraybackslash}X}

\providecommand*{\Xmath}[1]{\ensuremath{#1}\xspace}

\providecommand*{\XmathThreePar}[3]{%
  \ifthenelse{\equal{#3}{'}}{%
    \Xmath{#1_{#2}'}
  }{%
    \Xmath{#1_{#2}^{#3}} 
  }
}

\providecommand*{\XmathFourPar}[4]{\XmathThreePar{#1{#2}}{#3}{#4}}
\newcommand*{\Xop}[1]{\ensuremath{\mathrm{#1}\,}}

\newcommand*{\NSet}[3]{\XmathFourPar{\mathbb}{#1}{#2}{#3}}
\newcommandtwoopt*{\R}[2][][]{\NSet{R}{#1}{#2}}

\renewcommand*{\Re}{\Xop{Re}}
\renewcommand*{\Im}{\Xop{Im}}
\newcommand*{\Abs}[2][]{\XBrace[#1]{\lvert}{\rvert}{#2}}
\newcommand*{\XBrace}[4][]{\Xmath{\mathopen#1#2#4\mathclose#1#3}}

\newcommand*{\cframe}[1]{\Xmath{\mathcal\!{#1}}}

\newcommand*{\Quaternion}[3]{\XmathFourPar{}{#1}{#2}{#3}}
\newcommand*{\RealPart}[3]{\XmathFourPar\Re{#1}{#2}{#3}}
\newcommand*{\ImaginaryPart}[1]{\Im{#1}}

\newcommandtwoopt*{\qa}[2][][]{\Quaternion{a}{#1}{#2}}
\newcommandtwoopt*{\qb}[2][][]{\Quaternion{b}{#1}{#2}}
\newcommandtwoopt*{\qc}[2][][]{\Quaternion{c}{#1}{#2}}
\newcommandtwoopt*{\qd}[2][][]{\Quaternion{d}{#1}{#2}}
\newcommandtwoopt*{\qi}[2][][]{\Quaternion{i}{#1}{#2}}
\newcommandtwoopt*{\qj}[2][][]{\Quaternion{j}{#1}{#2}}
\newcommandtwoopt*{\qk}[2][][]{\Quaternion{k}{#1}{#2}}
\newcommandtwoopt*{\qn}[2][][]{\Quaternion{n}{#1}{#2}}
\newcommandtwoopt*{\qp}[2][][]{\Quaternion{p}{#1}{#2}}
\newcommandtwoopt*{\qs}[2][][]{\Quaternion{s}{#1}{#2}}
\newcommandtwoopt*{\qu}[2][][]{\Quaternion{u}{#1}{#2}}
\newcommandtwoopt*{\qq}[2][][]{\Quaternion{q}{#1}{#2}}
\newcommandtwoopt*{\qw}[2][][]{\Quaternion{w}{#1}{#2}}
\newcommandtwoopt*{\qx}[2][][]{\Quaternion{x}{#1}{#2}}
\newcommandtwoopt*{\qy}[2][][]{\Quaternion{y}{#1}{#2}}

\newcommandtwoopt*{\qnu}[2][][]{\Quaternion{\nu}{#1}{#2}}

\newcommandtwoopt*{\rea}[2][][]{\RealPart{a}{#1}{#2}}
\newcommandtwoopt*{\reb}[2][][]{\RealPart{b}{#1}{#2}}
\newcommandtwoopt*{\rec}[1][][]{\RealPart{c}{#1}}
\newcommandtwoopt*{\req}[1][][]{\RealPart{q}{#1}}
\newcommandtwoopt*{\rex}[1][][]{\RealPart{x}{#1}}

\newcommandtwoopt*{\ima}[2][][]{\Quaternion{\ImaginaryPart{a}}{#1}{#2}}
\newcommandtwoopt*{\imb}[2][][]{\Quaternion{\ImaginaryPart{b}}{#1}{#2}}
\newcommandtwoopt*{\imc}[2][][]{\Quaternion{\ImaginaryPart{c}}{#1}{#2}}
\newcommandtwoopt*{\imp}[2][][]{\Quaternion{\ImaginaryPart{p}}{#1}{#2}}
\newcommandtwoopt*{\imq}[2][][]{\Quaternion{\ImaginaryPart{q}}{#1}{#2}}
\newcommandtwoopt*{\imx}[2][][]{\Quaternion{\ImaginaryPart{x}}{#1}{#2}}
\newcommandtwoopt*{\imy}[2][][]{\Quaternion{\ImaginaryPart{y}}{#1}{#2}}

\if\ArxivFormat1
  \newcommand*{\NegSpace}{\!\!}
\else
  \newcommand*{\NegSpace}{}
\fi

\if\ArxivFormat1
  \title[Closed-form Solution of Wahba's Problem]{Closed-form Solution of Wahba's Problem for Pairwise Similar
  Quaternions}
  
  \author{Hristina Radak, Christian Scheunert, and Frank H. P. Fitzek}
  \address{H.~Radak and F.~Fitzek are with the Deutsche Telekom
  Chair of Communication Networks, Dresden University of Technology, 01062
  Dresden, Germany, Email: hristina.radak@tu-dresden.de}
  \address{C.~Scheunert is with the Chair of Information Theory and Machine
  Learning, Dresden University of Technology, 01062 Dresden, Germany}
  \address{F.~Fitzek is also with the Centre for Tactile
  Internet with Human-in-the-Loop (CeTI), Dresden University of Technology,
  01062 Dresden, Germany}
\else 
  \title{Closed-form Solution of Wahba's Problem for Pairwise Similar
  Quaternions}
  
  \author{Hristina Radak}
  \affil{Deutsche Telekom Chair of Communication Networks, Dresden University of
  Technology, 01062 Dresden, Germany}
  \author{Christian Scheunert}
  \affil{Chair of Information Theory and Machine Learning, Dresden University of
  Technology, 01062 Dresden, Germany}
  \author{Frank H. P. Fitzek}
  \affil{Deutsche Telekom Chair of Communication Networks and Centre for Tactile
  Internet with Human-in-the-Loop (CeTI), Dresden University of Technology,
  01062 Dresden, Germany}
\fi

\begin{document}

\maketitle


\if\ArxivFormat1
\begin{abstract}
Wahba's problem is fundamental to spacecraft attitude estimation, seeking the optimal rotation that minimizes the weighted misalignment between sets of vector observations. Traditional solvers, including Davenport's $q$-method, QUEST, and ESOQ, reformulate the problem as an eigenvalue task for a $4 \times 4$ symmetric matrix, a process that obscures the underlying algebraic structure of the solution. This paper presents a novel, entirely quaternion-based closed-form solution for the pairwise similar quaternions. By establishing a direct connection to the homogeneous singular Sylvester equation: \text{(i)} we derive the necessary and sufficient condition for the existence of a quaternion that achieves zero Wahba's cost; \text{(ii)} we provide a closed-form analytic expression for the corresponding solution set; and \text{(iii)} we propose the computationally efficient and numerically stable Minimal Analytic Rotation Algorithm (MARA). Computational complexity analysis demonstrates that MARA achieves a $35.11\%$ reduction in total floating-point operations (FLOPs) compared to the state-of-the-art ESOQ2 algorithm. Numerical validation via $10^6$ Monte Carlo trials confirms that MARA achieves higher accuracy than established optimal solvers under stochastic noise, offering a computationally more efficient and analytically transparent alternative for high-frequency attitude determination systems.

\end{abstract}
\fi

\if\ArxivFormat1
  \medskip
  \noindent\textbf{Keywords:} Quaternions; Wahba’s problem; Pairwise similarity; Closed-form solution;

\else 
%
%
%
%
%
%
%
%
%
\fi


\section{Introduction}

Determining the rotation between two coordinate frames is a fundamental problem in applied mathematics, with applications in spacecraft attitude estimation, robotics, and navigation. These fields require robust, accurate, and computationally efficient methods for recovering an object's rotation relative to another coordinate frame. Solutions often rely on vector observations expressed in both coordinate frames. The problem of finding the rotation between two sets of vector observations is formally known as Wahba's problem~\cite{wahba1965_least}. Wahba's problem has been addressed using various mathematical formulations, including rotation matrices and quaternions. Among these, quaternion-based approaches are particularly effective due to their compact representation, absence of singularities, and computational efficiency.

Let $\cframe{A}$, $\cframe{B}$ be two coordinate frames and $\{\qa[1], \qa[2], \dots, \qa[n]\}$, $\{\qb[1], \qb[2], \dots, \qb[n]\}$ denote the sets of nonzero pure quaternions, where $\qa[\ell]$, $\qb[\ell]$ with $\ell \in \{1,2,\dots,n\}$, $n \geq 2,$ represent vectors observed in $\cframe{A}$, $\cframe{B}$, respectively. Wahba's problem seeks a quaternion $\qq$ that minimizes the cost function
    \begin{align}
   	 	W_n(\qq) = \sum_{\ell=1}^{n} w_\ell \Abs{\qq[][-1]\qa[\ell]\qq - \qb[\ell]}^2,
    		\label{eq:WAHBAS_COST_FUNCTION_GENERAL}
    \end{align}
where the coefficients $w_\ell$ with $\ell \in \{1,2,\dots,n\}$ are positive real numbers. Without loss of generality, we may assume $w_\ell = 1$ for all $\ell$, as positive weights can be absorbed into $\qa[\ell]$ and $\qb[\ell]$.
Several well-established algorithms have been developed to solve Wahba's problem for the general case $n \geq 2$. Davenport's $\qq$-method~\cite{davenport1968_vector} reformulates the problem in~\eqref{eq:WAHBAS_COST_FUNCTION_GENERAL} as an eigenvalue problem involving a symmetric $4\times4$ matrix and computes a largest eigenvalue that yields the optimal quaternion. The QUEST algorithms~\cite{shuster1978_algorithm, shuster1978_approximat, shuster+oh1981_three-axis} further refine this approach by reducing the problem in~\eqref{eq:WAHBAS_COST_FUNCTION_GENERAL} to finding the largest root of a fourth-order polynomial via iterative methods. Later, works of Horn~\cite{horn1987_closed-for} and Markley~\cite{markley1988_attitude} proposed an alternative formulation employing the Singular Value Decomposition (SVD) of the attitude profile matrix to construct the optimal quaternion. Mortari~\cite{mortari1997_esoq} introduced the ESOQ method, which provides an analytic solution to the fourth-order polynomial, thereby eliminating iterative procedures. Markley and Mortari~\cite{markley+mortari2000_quaternion} subsequently improved ESOQ through the ESOQ2 algorithm, which reduces the number of arithmetic operations and enhances numerical stability relative to both ESOQ and QUEST. Wu~et~al.~\cite{WuZhou2018_FLAE} introduced the Fast Linear Attitude
Estimator~(FLAE), which employs an eigenvalue-based solution to obtain the optimal quaternion and to achieve higher computational speed compared to prior methods. However, no explicit floating-point operation count was reported to substantiate this claim.
The special case of $n = 2$ observations, however, is often treated separately. The TRIAD algorithm~\cite{shuster+oh1981_three-axis} offers a deterministic quaternion estimate for $n=2$, however, it is inherently suboptimal. Markley~\cite{markley2002_fast} explicitly addressed this suboptimality by introducing an algorithm that achieves the optimal solution while matching the computational speed of TRIAD.
In summary, all these algorithms for solving the original Wahba’s problem apply different methods of matrix calculus, which often obscure the underlying structure of solutions. To the best of our knowledge, no method for solving Wahba’s problem directly in the quaternion algebra has been published in the literature to date.

In this paper, we consider the important special case of pairwise similar quaternions $\{\qa[1], \qa[2], \dots, \qa[n]\}$, $\{\qb[1], \qb[2], \dots, \qb[n]\}$. In this case, there exists a quaternion $\qq$ such that Wahba’s cost function satisfies the condition $W_n(\qq)=0$. We show that only two arbitrary pairs of observed vectors $\{\qa[\ell],\qa[m]\}$, $\{\qb[\ell],\qb[m]\}$ with $\ell, m \in\{1,2,\dots,n\}$, $\ell\neq m$ are required to determine $\qq$.
We establish a direct algebraic relation between Wahba's cost function in~\eqref{eq:WAHBAS_COST_FUNCTION_GENERAL} and the homogeneous singular Sylvester equation in quaternions $\qa\qq - \qq\qb = 0$ studied in our companion work \cite{radak+scheunert++2025_on}. This relation leads to the three main contributions of this work: \text{(i)} we derive the necessary and sufficient condition for the existence of a quaternion that achieves zero Wahba's cost; \text{(ii)} we provide a closed-form analytic expression for the corresponding solution set; and \text{(iii)} we propose the computationally efficient and numerically stable Minimal Analytic Rotation Algorithm (MARA). Our approach solves the problem entirely within the quaternion algebra, avoiding conventional matrix calculus such as iterative eigenvalue computation or matrix inversion, thereby improving computational efficiency and analytical clarity of the solution.
Furthermore, numerical experiments demonstrate that the proposed method is robust to slight violations of the rigid pairwise similarity constraint caused by measurement noise or round-off errors.

This paper is organized as follows: Section~\ref{sec:PRELIMINARIES} reviews necessary quaternion preliminaries, defines the notion of pairwise similarity, and provides essential results regarding the homogeneous singular Sylvester equation in quaternions. Section~\ref{sec:SOLUTION_WAHBAS_PROBLEM} presents the main theoretical contributions of this paper, that is the existence criterion for zero Wahba's cost, the complete closed-form expression of the solution set in quaternions, and the Minimal Analytic Rotation Algorithm (MARA). In Section~\ref{sec:performance} we analyse the computational efficiency and the numerical stability of MARA and benchmark it against ESOQ2~\cite{mortari2000_second}, Markley's method~\cite{markley2002_fast}, and FLAE~\cite{WuZhou2018_FLAE}. Finally, Section~\ref{sec:CONCLUSION} concludes this work.


\section{Preliminaries}\label{sec:PRELIMINARIES}
\subsection{Quaternion Basics}
The set of quaternions forms a division ring over the real numbers with basis $\{1, \qi, \qj, \qk\}$. The imaginary units satisfy the relations $\qi[][2] = \qj[][2] = \qk[][2] = \qi\qj\qk = -1$, $\qi\qj = -\qj\qi = \qk$, $\qj\qk = -\qk\qj = \qi$, and $\qk\qi = -\qi\qk = \qj$.
A quaternion $\qa= \qa[w]+\qa[x]\qi+\qa[y]\qj+\qa[z]\qk$ can be written as
$\qa=\rea + \ima$, where $\rea=\qa[w]$ is the real scalar part of $\qa$
and $\ima = \qa[x]\qi+\qa[y]\qj+\qa[z]\qk$ is the vector imaginary
part of $\qa$. The conjugate of a quaternion $\qa$ is defined as
$\qa[][*] = \rea - \ima$. The norm of a quaternion $\qa$ is defined as 
$\Abs{\qa} = \sqrt{\smash[b]{\qa\qa[][*]}} = \sqrt{\smash[b]{\qa[][*]\!\qa}}
= \sqrt{\smash[b]{\qa[w][2] + \qa[x][2] + \qa[y][2] + \qa[z][2]}}$. A quaternion
$\qa$ is called unit if $\Abs{\qa} = 1$ and pure if $\rea=0$. 
By an extension of Euler's formula 
a quaternion $\qa$
can also be written as
$\qa = e^{\frac{\theta}{2}\qb} = \cos (\tfrac{\theta}{2}) + \qb \sin (\tfrac{\theta}{2})$, where $\qb$ is a pure quaternion that represents a rotation axis and $\theta$ is a rotation angle around axis $\qb$.
The inverse of a nonzero quaternion $\qa$ is given by $\qa[][-1] = \qa[][*] /\Abs{\qa}^2$.
Given two quaternions $\qa, \qb$ the quaternion product $\qa\qb$ is defined as
	\begin{align*}
    		\qa\qb = (\rea)(\reb) - \ima \! \cdot \imb + (\rea)(\imb) + (\reb)(\ima) + \ima \times \imb,
    \end{align*}
where $(\,\cdot\,, \times)$ denote the scalar and cross product, respectively. The quaternion product satisfies the following properties
\begin{align*}
\Abs{\qa\qb} = \Abs{\qa}\kern0.1em\Abs{\qb}, \qquad 
\Re{(\qa\qb)} = \Re{(\qb\qa)}, \qquad
(\qa\qb)^{-1} = \qb[][-1]\qa[][-1],
\end{align*}
and the squared norm of the sum of quaternions satisfies
\begin{align*}
	\Abs{\qa \pm \qb}^2 = \Abs{\qa}^2 \pm 2 \Re(\qa \qb[][*]) + \Abs{\qb}^2.
\end{align*}
In particular, if $\qa, \qb$ are pure quaternions, then it holds
\begin{align*}
\qa[][2] = -\qa\cdot\qa=-\Abs{\qa}^2, \quad \qa\qb = -\qa \cdot \qb + \qa \times \qb, \quad \qa\qb+\qb\qa=2\Re(\qa\qb).
\end{align*}

\begin{lemma}\label{lemma:AXIS_PARALLEL_QUATERNIONS}
For nonreal quaternions $\qa, \qb$ it holds $\qa\qb = \qb\qa$ if and only if 
there exist real numbers $\mu,\lambda$ such that $\qb = \mu + \lambda(\ima)$.
\end{lemma}

\begin{proof}
Since $\Re{(\qa\qb)} = \Re{(\qb\qa)}$ it holds $\Re{(\qa\qb)} - \Re{(\qb\qa)} = 0$ such that
\begin{align*}
	\qa\qb - \qb\qa = \Im(\qa\qb) - \Im(\qb\qa) = \ima \times \imb - \imb \times \ima = 2\,(\ima\times\imb),
\end{align*}
that is $\qa\qb = \qb\qa$ if and only if $\ima = \lambda(\imb)$, which yields the assertion of the lemma.
\end{proof}

\begin{lemma}\label{lemma:AC_EQUAL_MINUS_CA}
For nonzero pure quaternions $\qa,\qb$, and $\qc = \qa\times\qb$, it holds $\qa\qc = -\qc\qa$.
\end{lemma}

\begin{proof}
Since $\qa\cdot\qc = \qa\cdot(\qa\times\qb)= 0$, we have $\qa\qc = \qa\times\qc = -\qc\times\qa = -\qc\qa$.
\end{proof}

\begin{definition}\label{def:SIMILAR_QUATERNIONS}
Nonreal quaternions $\qa,\qb$ are called similar if there exists a nonzero
quaternion $\qp$ such that $\qp[][-1]\qa\qp = \qb$. In that case, we write
$\qa \sim \qb$.
\end{definition}

\begin{lemma}[\NegSpace\cite{radak+scheunert++2025_on}]
\label{lemma:SIMILAR_QUATERNIONS_CHARACTERIZATION}
Nonreal quaternions $\qa,\qb$ are similar if and only if $\rea = \reb$ and
$\Abs{\qa} = \Abs{\qb}$.
\end{lemma}

\subsection{Pairwise Similar Quaternions}\label{sec:PAIRWISE_SIMILAR_QUATERNIONS}

\begin{definition}\label{def:PAIRWISE_SIMILAR_QUATERNIONS}
Sets of nonreal quaternions $\{\qa[1],\qa[2],\dots,\qa[n] \}$, $\{\qb[1],\qb[2],\dots,\qb[n]\}$ with $n\geq 2$ and $\ima[\ell]\neq\ima[m]\,$ for all $\ell,m \in \{1,2,\dots,n\}$, $\ell\neq m$, are called pairwise similar if there exists a nonzero quaternion $\qp$ such that $\qp[][-1]\qa[\ell]\qp = \qb[\ell]\,$ for all $\ell \in \{1, 2, \dots, n\}$. In that case we write $(\qa[1],\qa[2],\dots, \qa[n]) \sim (\qb[1],\qb[2],\dots, \qb[n])$ and denote the set of all such quaternions $\qp$ by $Q_n$.
\end{definition}

\begin{lemma}\label{lemma:REDUCTION_SIMILARITY_N_TO_2}
For sets of nonreal quaternions $\{\qa[1],\qa[2],\dots,\qa[m] \}$, $\{\qb[1],\qb[2],\dots,\qb[m]\}$ that satisfy $(\qa[1],\qa[2],\dots, \qa[m]) \sim (\qb[1],\qb[2],\dots, \qb[m])$ for all $m\in\{1,2,\dots,n\}$, $n\geq 2$, it holds $Q_1 \supset Q_2 = Q_3 = \cdots = Q_n$.
\end{lemma}

\begin{proof}
By definition, it holds $Q_1 \supseteq Q_2 \supseteq \cdots \supseteq Q_n$ since for all $m \in \{2,3,\dots,n\}$ each $Q_m$ imposes the constraints of $Q_{m-1}$ and the additional constraint $\qp[][-1]\qa[m]\qp = \qb[m]$. Note, since $(\qa[1],\qa[2],\dots, \qa[n]) \sim (\qb[1],\qb[2],\dots, \qb[n])$ the set $Q_n$ is not empty. Therefore, it suffices to show $Q_2 \subseteq Q_3 \subseteq \dots \subseteq Q_n$ and $Q_1 \supset Q_2$.
Let $\qp \in Q_m$ with $m \in \{3,4,\dots,n\}$. Since $Q_{m-1} \supseteq Q_m$ it holds $\qp \in Q_{m-1}$. Further, there exists a nonzero quaternion $\qq$ such that $\qp\qq \in Q_{m-1}$ and hence
\begin{align*}
	\qb[\ell] = (\qp\qq)^{-1} \qa[\ell](\qp\qq) = \qq[][-1]\qp[][-1]\qa[\ell]\qp\qq = \qq[][-1]\qb[\ell]\qq, \quad \ell\in\{1,2,\dots,m-1\}.
\end{align*}
Since $(\qa[1],\qa[2],\dots, \qa[n]) \sim (\qb[1],\qb[2],\dots, \qb[n])$, by Lemma~\ref{lemma:AXIS_PARALLEL_QUATERNIONS} there exist real numbers $\mu_{\ell},\lambda_{\ell}$ that satisfy $\qq = \mu_{\ell} + \lambda_{\ell}(\imb[\ell])$ for all $\ell \in\{1,2,\dots, m-1\}$. Therefore, we must have $\mu_{1} = \mu_{2} = \dots = \mu_{m-1} \neq 0$ and $\lambda_{1} = \lambda_{2} = \dots = \lambda_{m-1} = 0$ such that $\qq$ is a nonzero real number. Then it holds
\begin{align*}
	(\qp\qq)^{-1} \qa[m](\qp\qq) = \qq[][-1]\qp[][-1]\qa[m]\qp\qq = \qq[][-1]\qb[m]\qq = \qb[m]
\end{align*}
such that $\qp\qq \in Q_{m}$, that is $Q_{m-1} = Q_m$. By induction we obtain $Q_2 = Q_3 = \cdots = Q_n$.
Now let $\qp\in Q_2$ and hence $\qp\in Q_1$. Further, for arbitrary real numbers $\mu,\lambda$ let $\qq = \mu + \lambda(\imb[1])$. Then, by Lemma~\ref{lemma:AXIS_PARALLEL_QUATERNIONS} we have 
\begin{align*}
  \qq\qb[1] = \qb[1]\qq \quad \Longleftrightarrow \quad \qb[1] = \qq[][-1]\qb[1]\qq = \qq[][-1]\qp[][-1]\qa[1]\qp\qq = (\qp\qq)^{-1}\qa(\qp\qq),
\end{align*}
that is $\qp\qq \in Q_1$. For $\lambda \neq 0$ the quaternion $\qq$ is nonreal. In that case we have $\imq = \lambda(\imb[1]) \neq \nu(\imb[2])$ for all real numbers $\nu$ such that
\begin{align*}
	(\qp\qq)^{-1} \qa[2](\qp\qq) = \qq[][-1]\qp[][-1]\qa[2]\qp\qq = \qq[][-1]\qb[2]\qq \neq \qb[2]
\end{align*}
and hence $\qp\qq \notin Q_2$, that is $Q_1 \supset Q_2$.
\end{proof}

\begin{lemma}\label{lemma:PAIRWISE_SIMILAR_QUATERNIONS_CHARACTERIZATION}
For nonreal quaternions $\{\qa[1],\qa[2]\},\{\qb[1],\qb[2]\}$ it holds $(\qa[1],\qa[2]) \sim (\qb[1],\qb[2])$ if and only if $\qa[1] \sim \qb[1]$, $\qa[2] \sim \qb[2]$, and $\Re(\qa[1]\qa[2]) = \Re(\qb[1]\qb[2])$.
\end{lemma}

\begin{proof}
First assume that $(\qa[1],\qa[2]) \sim (\qb[1],\qb[2])$. Then, there exists a nonzero quaternion $\qp$ such that $\qp[][-1]\qa[1]\qp = \qb[1]$ and $\qp[][-1]\qa[2]\qp = \qb[2]$. It follows that $\qp[][-1](\qa[1] + \qa[2])\qp = \qb[1] + \qb[2]$. Hence, we have $\qa[1] \sim \qb[1]$, $\qa[2] \sim \qb[2]$, and $(\qa[1] + \qa[2]) \sim (\qb[1] + \qb[2])$.
Then, by Lemma~\ref{lemma:SIMILAR_QUATERNIONS_CHARACTERIZATION} we obtain $\Abs{\qa[1]} = \Abs{\qb[1]}$, $\Abs{\qa[2]} = \Abs{\qb[2]}$, and $\Abs{\qa[1] + \qa[2]} = \Abs{\qb[1] + \qb[2]}$ such that
	\begin{align*}
	    \Abs{\qa[1]}^2 + 2\Re(\qa[1]\qa[2][*]) + \Abs{\qa[2]}^2 = \Abs{\qa[1] + \qa[2]}^2 = \Abs{\qb[1] + \qb[2]}^2 = \Abs{\qb[1]}^2 + 2\Re(\qb[1]\qb[2][*]) + \Abs{\qb[2]}^2,
	\end{align*}
which yields $\Re(\qa[1]\qa[2]) = \Re(\qb[1]\qb[2])$. Conversely, assume $\qa[1] \sim \qb[1]$, $\qa[2] \sim \qb[2]$, and $\Re(\qa[1]\qa[2]) = \Re(\qb[1]\qb[2])$. Then, by Lemma~\ref{lemma:SIMILAR_QUATERNIONS_CHARACTERIZATION} it holds
\begin{align*}
	\Re(\qa[1] + \qa[2]) = \rea[1] + \rea[2] = \reb[1] + \reb[2] = \Re(\qb[1] + \qb[2])
\end{align*}
and further 
\begin{align*}
\Abs{\qa[1] + \qa[2]}^2 = \Abs{\qa[1]}^2 + 2\Re(\qa[1]\qa[2][*]) + \Abs{\qa[2]}^2 = \Abs{\qb[1]}^2 + 2\Re(\qb[1]\qb[2][*]) + \Abs{\qb[2]}^2 = \Abs{\qb[1] + \qb[2]}^2,
\end{align*}
such that $(\qa[1] + \qa[2]) \sim (\qb[1] + \qb[2])$. Hence, there exists a nonzero quaternion $\qq$ such that
\begin{align}
	\qq[][-1](\qa[1] + \qa[2])\qq = \qb[1] + \qb[2].
	\label{eq:SUM_SIMILARITY}
\end{align}
Next, assume that $\qq[1]$ is a nonzero quaternion with $\qq[1][-1]\qa[1]\qq[1]=\qb[1]$. Set $\qq = \qq[1]\qq[2]$, where $\qq[2]$ is a nonzero quaternion with $\qb[1]\qq[2] = \qq[2]\qb[1]$. Then we have
\begin{align}
	\qq[][-1]\qa[1]\qq[]= (\qq[1]\qq[2])^{-1}\qa[1]\qq[1]\qq[2] = \qq[2][-1]\qq[1][-1]\qa[1]\qq[1]\qq[2] = \qq[2][-1]\qb[1]\qq[2] = \qq[2][-1]\qq[2]\qb[1] = \qb[1]. 
	\label{eq:PAIRWISE_SIMILARITY_PROOF_PART_1}
\end{align}
Substituting \eqref{eq:PAIRWISE_SIMILARITY_PROOF_PART_1} into \eqref{eq:SUM_SIMILARITY} yields $\qq[][-1]\qa[2]\qq[] = \qb[2]$, which completes the proof.
\end{proof}


\subsection{Singular Sylvester equation in quaternions}\label{sec:SSE}
When attempting to solve Wahba's problem, constraints resulting from the noncommutativity of the quaternion product imply
that the solutions are related to the homogeneous singular Sylvester equation in quaternions
\begin{align}
\qa\qq - \qq\qb = 0.
\label{eq:HOMOGENEOUS_SINGULAR_SYLVESTER_EQUATION}
\end{align}
We summarize the results from~\cite{radak+scheunert++2025_on} and recall the necessary and sufficient condition for solvability, as well as the general solution of~\eqref{eq:HOMOGENEOUS_SINGULAR_SYLVESTER_EQUATION}, both of which are fundamental to the development of our main results presented in Section~\ref{sec:SOLUTION_WAHBAS_PROBLEM}.
\begin{lemma}[\NegSpace\cite{radak+scheunert++2025_on}]
\label{lemma:QUATERNION_SQRT}
Let $\qa$ be a nonzero quaternion. Then, $\qq[][2] - \qa = 0$ has the
solutions
\vspace*{-1ex}
\begin{align*}
  \qq = \pm \sqrt{\Abs{\qa}} \, \frac{\qp}{\Abs{\qp}},
\end{align*}
where $\qp$ is either an arbitrary nonzero pure quaternion if $\qa$ is a negative real
number, or $\qp = \qa + \Abs{\qa}$ if $\qa$ is a nonreal quaternion or a positive
real number.
\end{lemma}
\begin{lemma}[\NegSpace\cite{radak+scheunert++2025_on}]
\label{lemma:EXISTENCE_OF_SOLUTIONS_OF_THE_HOMOGENEOUS_SYLVESTER_EQUATION}
Let $\qa, \qb$ be nonreal quaternions. Then, 
\eqref{eq:HOMOGENEOUS_SINGULAR_SYLVESTER_EQUATION} has a nonzero solution if and only if
$\qa \sim \qb$.
\end{lemma}
\begin{lemma}[\NegSpace\cite{radak+scheunert++2025_on}]
\label{lemma:GENERAL_SOLUTION_OF_THE_HOMOGENEOUS_SYLVESTER_EQUATION}

Let $\qa,\qb$ be similar quaternions and $\qp$ an arbitrary quaternion. Then, the general
solution to \eqref{eq:HOMOGENEOUS_SINGULAR_SYLVESTER_EQUATION} is given by
\begin{align*}
  \qq = (\ima)\qp + \qp(\imb).
\end{align*}

\end{lemma}
\begin{theorem}[\NegSpace\cite{radak+scheunert++2025_on}]
\label{theorem:NONZERO_SOLUTIONS_OF_THE_HOMOGENEOUS_SYLVESTER_EQUATION}
Let $\qa,\qb$ be similar quaternions and $\lambda, \mu$ arbitrary real numbers with $\Abs{\lambda} + \Abs{\mu\,\Im{(\qa + \qb)}} \neq 0$. Then, the nonzero solutions to \eqref{eq:HOMOGENEOUS_SINGULAR_SYLVESTER_EQUATION} are given by
\begin{align}
  \qq = \lambda\sqrt{(\ima)(\imb[][*])} + \mu\,\Im{(\qa + \qb)},
  \label{eq:NONZERO_SOLUTIONS_OF_THE_HOMOGENEOUS_SYLVESTER_EQUATION}
\end{align}
subject to $\ima\cdot\imq = 0$ if $\ima =- \imb$.
\end{theorem}


\section{Solutions of Wahba's Problem for Pairwise Similar Quaternions}\label{sec:SOLUTION_WAHBAS_PROBLEM}

We now consider the problem in which Wahba's cost function $W_n$ for $n \geq 2$ is to be minimized
for the special case of pairwise similar quaternions
$\{\qa[1],\qa[2],\dots,\qa[n] \}$, $\{\qb[1],\qb[2],\dots,\qb[n]\}$.
In this case, there exist a nonreal quaternion $\qq$ such that
\begin{align}
  \qa[\ell] \sim \qb[\ell] 
  \quad \Longleftrightarrow \quad \qq[][-1]\qa[\ell]\qq - \qb[\ell] = 0
  \quad \Longleftrightarrow \quad \Abs{\qq[][-1]\qa[\ell]\qq - \qb[\ell]}^2 = 0
  \qquad \text{for all $\ell \in \{1,2,\dots,n\}$}.
  \label{eq:GENERAL_SOLUTION_SINGLE_WAHBAS_COST_FUNCTION}
\end{align}
The problem therefore consists of finding a quaternion $\qq$ such that $W_n(\qq) = 0$. In Lemma~\ref{lemma:REDUCTION_SIMILARITY_N_TO_2}, we have shown that it is
sufficient to consider arbitrary subsets $\{\qa[1],\qa[2]\}$,
$\{\qb[1],\qb[2]\}$ of $\{\qa[1],\qa[2],\dots,\qa[n] \}$, 
$\{\qb[1],\qb[2],\dots,\qb[n]\}$ with only two pairwise similar quaternions.
Referring to this result, it is sufficient to consider the problem
\begin{align}
	W_2(\qq) = \sum_{\ell= 1}^2 \Abs{\qq[][-1]\qa[\ell]\qq - \qb[\ell]}^2 = 0.
	\label{eq:WAHBAS_EQUATION}
\end{align}
Then, the resulting quaternion $\qq$ is also the solution to $W_n(\qq) = 0$.
Moreover, for similar quaternions $\qa[\ell],\qb[\ell]$ it holds 
$\reb[\ell] = \rea[\ell] = (\rea[\ell])\qq[][-1]\qq = \qq[][-1](\rea[\ell])\qq$ such that
\begin{align*}
  \qq[][-1]\qa[\ell]\qq - \qb[\ell] 
  = \qq[][-1](\rea[\ell])\qq  + \qq[][-1](\ima[\ell])\qq - \reb[\ell] - \imb
  = \qq[][-1](\ima[\ell])\qq - \imb[\ell].
\end{align*}
Hence, without loss of generality, we may consider the quaternions 
$\qa[\ell], \qb[\ell]$ to be pure. We will make use of this assumption
in proofs of the following theorems. However, the statements of the theorems
will be presented in their general form.

\subsection{Analytical Derivation of the Solution Set}\label{sec:DERIVATION}

\begin{theorem}\label{theorem:EXISTENCE_OF_SOLUTIONS_OF_WAHBAS_EQUATION}
Let $\{\qa[1],\qa[2]\},\{\qb[1],\qb[2]\}$ be sets of nonreal quaternions. Then, \eqref{eq:WAHBAS_EQUATION} has a solution if and only if $(\qa[1],\qa[2]) \sim (\qb[1], \qb[2])$.
\end{theorem}

\begin{proof}
Assume that \eqref{eq:WAHBAS_EQUATION} holds. Then, there exists a nonzero quaternion $\qq$ such that $\qq[][-1]\qa[1]\qq = \qb[1]$ and $\qq[][-1]\qa[2]\qq = \qb[2]$, that is $(\qa[1],\qa[2]) \sim (\qb[1], \qb[2])$. Conversely, assume $(\qa[1],\qa[2]) \sim (\qb[1], \qb[2])$. Then, there exists a nonzero quaternion $\qp$ such that $\qp[][-1]\qa[1]\qp = \qb[1]$ and $\qp[][-1]\qa[2]\qp = \qb[2]$. Hence, $\qq = \qp$ is a solution of \eqref{eq:WAHBAS_EQUATION}.
\end{proof}

\begin{lemma}\label{lemma:HOMOGENEOUS_WAHBAS_COST_FUNCTION_WITH_PERPENDICULAR_COMPONENTS}
Let $\{\qa[1],\qa[2]\},\{\qb[1],\qb[2]\}$ be sets of nonreal quaternions with $(\qa[1],\qa[2]) \sim (\qb[1], \qb[2])$. Then, \eqref{eq:WAHBAS_EQUATION} is equivalent to
\begin{align}
	\Abs{\qq[][-1](\ima[1])\qq - \imb[1]}^2 + \Abs{\qq[][-1](\ima[1]\times\ima[2])\qq - \imb[1]\times\imb[2]}^2 = 0.
	\label{eq:HOMOGENEOUS_WAHBAS_EQUATION_PERPENDICULAR_COMPONENTS}
\end{align}
\end{lemma}

\begin{proof}
We may, without loss of generality, consider $\qa[1], \qa[2], \qb[1], \qb[2]$ to be pure. Let $\qq$ be a solution of \eqref{eq:WAHBAS_EQUATION}. Then it equivalently holds $\qq[][-1]\qa[1]\qq = \qb[1]$ and $\qq[][-1]\qa[2]\qq = \qb[2]$. Hence, $\Abs{\qq[][-1]\qa[1]\qq - \qb[1]}^2 = 0$. Multiplying $\qq[][-1]\qa[2]\qq = \qb[2]$ by $\qq[][-1]\qa[1]\qq$ from the left and with $\Re(\qa[1]\qa[2]) = \Re(\qb[1]\qb[2])$, we obtain
\begin{align*}
	\qq[][-1]\qa[1]\qa[2]\qq = \qq[][-1]\qa[1]\qq\qb[2] = \qb[1]\qb[2] \quad \Longleftrightarrow \quad \qq[][-1](\qa[1]\times\qa[2])\qq = \qb[1]\times\qb[2],
\end{align*}
which completes the proof.
\end{proof}

\begin{theorem}\label{theorem:SOLUTION_TO_THE_HOMOGENEOUS_WAHBAS_EQUATION}
Let $\{\qa[1],\qa[2]\}$,$\{\qb[1],\qb[2]\}$ be sets of nonreal quaternions with $(\qa[1],\qa[2]) \sim (\qb[1], \qb[2])$. For arbitrary real numbers $\lambda_1, \lambda_2$, and $\mu_1$ with $\Abs{\lambda_1} + \Abs{\mu_1\,\Im{(\qa[1] + \qb[1])}} \neq 0$ and $\lambda_2 \neq 0$, the general solution of \eqref{eq:WAHBAS_EQUATION} is given by $\qq=\qq[1]\qq[2]$, where
\begin{align}
	\qq[1] &= \lambda_1\sqrt{(\ima[1])(\imb[1][*])} + \mu_1\,\Im{(\qa[1] + \qb[1])},\label{eq:SOLUTION_HOMOGENEOUS_WAHBAS_EQUATION_Q1}\\
	\qq[2] &= \lambda_2\sqrt{\qq[1][*](\ima[1]\times\ima[2])\qq[1](\imb[2]\times\imb[1])},
	\label{eq:SOLUTION_HOMOGENEOUS_WAHBAS_EQUATION_Q2}
\end{align}
subject to
\begin{align}
	\ima[1]\cdot\imq[1] = 0 \quad &\text{if} \quad \ima[1] =- \imb[1],\label{eq:SOLUTION_TO_THE_HOMOGENEOUS_WAHBAS_EQUATION_CONDITION_1}\\
	\qq[1][*](\ima[1]\times\ima[2])\qq[1]\cdot\imq[2] = 0 \quad &\text{if} \quad \qq[1][*](\ima[1]\times\ima[2])\qq[1] = -\imb[1]\times\imb[2].\label{eq:SOLUTION_TO_THE_HOMOGENEOUS_WAHBAS_EQUATION_CONDITION_2}
\end{align}
\end{theorem}

\begin{proof}
We may, without loss of generality, consider $\qa[1],\qa[2],\qb[1],\qb[2]$ to be pure. 
Let $\qq$ be a solution of \eqref{eq:WAHBAS_EQUATION}. Then, by 
Lemma~\ref{lemma:HOMOGENEOUS_WAHBAS_COST_FUNCTION_WITH_PERPENDICULAR_COMPONENTS}
it holds
\begin{align*}
		\Abs{\qq[][-1]\qa[1]\qq - \qb[1]}^2 + \Abs{\qq[][-1](\qa[1]\times\qa[2])\qq - \qb[1]\times\qb[2]}^2 = 0
\end{align*}
and hence
\begin{align}
		\Abs{\qq[][-1]\qa[1]\qq - \qb[1]}^2 = 0 \quad \text{and} \quad \Abs{\qq[][-1](\qa[1]\times\qa[2])\qq - \qb[1]\times\qb[2]}^2 = 0.
		\label{eq:SOLUTION_TWO_ALIGNMENTS}
\end{align}
Let $\qq[1]$ be a quaternion such that $\Abs{\qq[1][-1]\qa[1]\qq[1] - \qb[1]}^2 = 0$. 
Then, by \eqref{eq:GENERAL_SOLUTION_SINGLE_WAHBAS_COST_FUNCTION} and Theorem~\ref{theorem:NONZERO_SOLUTIONS_OF_THE_HOMOGENEOUS_SYLVESTER_EQUATION}
$\qq[1]$ is given by~\eqref{eq:SOLUTION_HOMOGENEOUS_WAHBAS_EQUATION_Q1} subject to $\Abs{\lambda_1} + \Abs{\mu_1(\qa[1] + \qb[1])} \neq 0$ and $\qa[1]\cdot\imq[1] = 0$ if $\qa[1] =- \qb[1]$. Further, it holds $\qa[1]\qq[1] = \qq[1]\qb[1]$. Next, for every nonzero quaternion $\qq$ there exists a unique nonzero quaternion $\qq[2]$ such that $\qq = \qq[1]\qq[2]$. Then, substituting $\qq = \qq[1]\qq[2]$ and $\qa[1]\qq[1] = \qq[1]\qb[1]$ into the left side equality in \eqref{eq:SOLUTION_TWO_ALIGNMENTS} yields
\begin{align}
	\qb[1] = (\qq[1]\qq[2])^{-1}\qa[1]\qq[1]\qq[2] = \qq[2][-1]\qq[1][-1]\qq[1]\qb[1]\qq[2] = \qq[2][-1]\qb[1]\qq[2] \quad \Longleftrightarrow \quad \qq[2]\qb[1] = \qb[1]\qq[2].
	\label{eq:HOMOGENEOUS_SOLUTION_PROOF_PART_1}
\end{align}
Further, substituting $\qq = \qq[1]\qq[2]$ into the right side equality in \eqref{eq:SOLUTION_TWO_ALIGNMENTS} yields
\begin{align}
	(\qq[1]\qq[2])^{-1}(\qa[1]\times\qa[2])\qq[1]\qq[2] - \qb[1]\times\qb[2] = \qq[2][-1]\qq[1][-1](\qa[1]\times\qa[2])\qq[1]\qq[2] - \qb[1]\times\qb[2] = 0.
	\label{eq:HOMOGENEOUS_SOLUTION_PROOF_PART_2}
\end{align}
Define $\qa[3]=\qq[1][-1](\qa[1]\times\qa[2])\qq[1]$ and $\qb[3] = \qb[1]\times\qb[2]$. Then $\rea[3] = \Re(\qa[1] \times\qa[2]) = 0$ and $\reb[3] = \Re(\qb[1] \times\qb[2]) = 0$. Since $(\qa[1],\qa[2]) \sim (\qb[1],\qb[2])$ it holds $\Abs{\qa[1]\qa[2]} = \Abs{\qb[1]\qb[2]}$ and hence
\begin{align*}
	\Abs{\qa[3]} = \Abs{\qq[1][-1](\qa[1]\times\qa[2])\qq[1]} = \Abs{\qa[1]\times\qa[2]} = \Abs{\qb[1]\times\qb[2]} = \Abs{\qb[3]}
\end{align*}
such that $\qa[3] \sim \qb[3]$. Moreover, by Lemma~\ref{lemma:AC_EQUAL_MINUS_CA} it holds $(\qa[1]\times\qa[2])\qa[1] = - \qa[1](\qa[1]\times\qa[2])$ and $\qb[3]\qb[1] = -\qb[1]\qb[3]$.  Then, by the left side equality in~\eqref{eq:SOLUTION_TWO_ALIGNMENTS} we have $\qb[1] = \qq[][-1]\qa[1]\qq$ and hence
\begin{align}
	\qa[3]\qb[1] = \qq[][-1](\qa[1]\times\qa[2])\qq\qq[][-1]\qa[1]\qq = - \qq[][-1]\qa[1]\qq\qq[][-1](\qa[1]\times\qa[2])\qq = -\qb[1]\qa[3].
	\label{eq:B3_ORTHOGONAL_B1}
\end{align}
Next, multiplying \eqref{eq:HOMOGENEOUS_SOLUTION_PROOF_PART_2} by $\qq[2]$ from the left yields $\qa[3]\qq[2] -\qq[2]\qb[3] = 0$ such that by Theorem~\ref{theorem:NONZERO_SOLUTIONS_OF_THE_HOMOGENEOUS_SYLVESTER_EQUATION} for arbitrary real numbers $\lambda_2, \mu_2$ with $\Abs{\lambda_2} + \Abs{\mu_2\,(\qa[3] + \qb[3])} \neq 0$ we obtain
	\begin{align}
		\qq[2] = \lambda_2\sqrt{\qa[3]\qb[3][*]} + \mu_2(\qa[3] + \qb[3]),
		\label{eq:SOLUTION_HOMOGENEOUS_WAHBAS_EQUATION_Q2_PROOF}
	\end{align}
subject to $\qa[3]\cdot\imq[2] = 0$ if $\qa[3] =- \qb[3]$. Substituting ~\eqref{eq:SOLUTION_HOMOGENEOUS_WAHBAS_EQUATION_Q2_PROOF} into the right side equality in~\eqref{eq:HOMOGENEOUS_SOLUTION_PROOF_PART_1}, and with~\eqref{eq:B3_ORTHOGONAL_B1} as well as $\qb[3]\qb[1] = -\qb[1]\qb[3]$ we obtain
\begin{align*}
	0 = \qq[2]\qb[1] - \qb[1]\qq[2] = \mu_2(\qa[3]\qb[1] + \qb[3]\qb[1] - \qb[1]\qa[3]- \qb[1]\qb[3]) = 2\mu_2(\qa[3] + \qb[3])\qb[1]
\end{align*}
such that $\mu_2(\qa[3] + \qb[3]) = 0$, which completes the proof.
\end{proof}

\subsection{Minimal Analytical Rotation Algorithm (MARA)}
\label{sec:MARA}
According to Theorem~\ref{theorem:EXISTENCE_OF_SOLUTIONS_OF_WAHBAS_EQUATION}, a solution to \eqref{eq:WAHBAS_EQUATION} exists if and only if $(\qa[1], \qa[2]) \sim (\qb[1], \qb[2])$. A necessary condition for the pairwise similarity is the preservation of the scalar part of the product, $\Re(\qa[1]\qa[2]) = \Re(\qb[1]\qb[2])$. Given that the quaternions are pure, that is $\rea[\ell] = \reb[\ell] = 0$ for $\ell\in\{1,2\}$, this condition reduces to the equality of the inner products
\begin{align}
\ima[1] \cdot \ima[2] = \imb[1] \cdot \imb[2].
\end{align}
Geometrically, this implies that the angle between $\ima[1], \ima[2]$ must be identical to the angle between $\imb[1], \imb[2]$. This condition constitutes a fundamental geometric constraint on the existence of an alignment between two sets of observations. 

While Theorem~\ref{theorem:SOLUTION_TO_THE_HOMOGENEOUS_WAHBAS_EQUATION} provides the complete algebraic solution set of \eqref{eq:WAHBAS_EQUATION}, a practical implementation requires a single solution. By fixing the free parameters $\lambda_1, \lambda_2$, and $\mu_1$ in Theorem~\ref{theorem:SOLUTION_TO_THE_HOMOGENEOUS_WAHBAS_EQUATION} and by Lemma~\ref{lemma:QUATERNION_SQRT}, we obtain MARA, as detailed in Algorithm~\ref{alg:MARA}.

\begin{algorithm}[ht]
\caption{Minimal Analytical Rotation Algorithm (MARA) for $n=2$}
\label{alg:MARA}
\begin{algorithmic}[1]
\State \textbf{Input:} Pure quaternions $\{\qa[1],\qa[2]\},\{\qb[1],\qb[2]\}$ with $(\qa[1], \qa[2]) \sim (\qb[1], \qb[2])$
\State \textbf{Output:} Optimal quaternion $\qq$
\State $\qq[1] \gets \qa[1] + \qb[1]$
\State $\qp \gets \qq[1][*] (\qa[1] \times \qa[2]) \qq[1] (\qb[2] \times \qb[1])$
\State $\qq[2] \gets \qp + |\qp|$
\State $\qq \gets \qq[1] \qq[2]$
\end{algorithmic}
\end{algorithm}

\begin{remark}
\label{remark:MARA_NEGLECTED_CONDITIONS}
The structure of MARA is based on results of
Theorem~\ref{theorem:SOLUTION_TO_THE_HOMOGENEOUS_WAHBAS_EQUATION} and
Lemma~\ref{lemma:QUATERNION_SQRT}, although several considerations must be
taken into account. In line~3 we set $\lambda_1 = 0$ and $\mu_1 = 1$ to minimize computational cost. Although Theorem~\ref{theorem:SOLUTION_TO_THE_HOMOGENEOUS_WAHBAS_EQUATION} formally requires a test for the case $\ima[1] = -\imb[1]$, this condition defines an event of measure zero and thus occurs with probability zero. Consequently, we omit this test. In practical implementations, exact equality of floating-point variables is neither expected nor numerically advisable due to rounding and representation errors. An analogous argument applies to line~5, where the computation of $\qq[2]$ assumes that $\qp$ in line~4 is nonreal. If $\qp$ were real, the cross products $(\qa[1] \times \qa[2])$ and $(\qb[1] \times \qb[2])$ would be parallel, which again constitutes an event of measure zero. Thus,
$\qp$ is nonreal with probability~$1$ and the algorithmic steps remain valid in both theory and practice.
\end{remark}

\section{Performance Analysis and Numerical Validation}
\label{sec:performance}

This section evaluates the computational complexity and numerical stability of the proposed MARA. Performance is benchmarked against the Second-Order Estimator of the Optimal Quaternion (ESOQ2) \cite{mortari2000_second}, Markley's closed-form solution \cite{markley2002_fast}, and the Fast Linear Attitude Estimator (FLAE) \cite{WuZhou2018_FLAE}. A C++ implementation of these algorithms as well as the complete test environment is provided by an open-source repository in~\cite{MARA_Repository}.

\subsection{Computational Complexity}

Computational complexity is evaluated by counting floating-point operations (FLOPs) in a controlled C++ environment. Table~\ref{table:benchmarks} summarizes the explicit addition, multiplication, division, and transcendental square root counts for each algorithm for the case $n=2$.

\begin{table}[hbt!]
\centering
\caption{Computational Complexity Profiles for Two-Vector Deterministic Solvers}
\label{table:benchmarks}
\small 
\renewcommand{\arraystretch}{1.3}
\begin{tabularx}{\textwidth}{X c c c c}
\toprule
{Operation} & {ESOQ2} & {Markley} & {FLAE} & {MARA} \\ 
\midrule
Additions & 58  & 38  & 105 & 34 \\
Multiplications        & 70  & 66  & 134 & 50 \\
Divisions              & 0   & 11  & 6   & 0  \\
Square Roots           & 3   & 4   & 0   & 1  \\ 
\midrule
{Total FLOPs}          & {131} & {119} & {245} & {85} \\ 
\bottomrule
\end{tabularx}
\end{table}

As shown in Table~\ref{table:benchmarks}, MARA is the most efficient optimal solver, that requires a total of just 85 FLOPs. This represents a $35.11\%$ reduction in FLOPs relative to ESOQ2, a $28.57\%$ reduction relative to Markley's formulation, and a $65.31\%$ reduction relative to FLAE. Crucially, MARA reduces the number of required square roots down to just one, representing a $66.67\%$ reduction compared to ESOQ2 and a $75.00\%$ reduction compared to Markley's formulation.

\subsection{Experimental Setup}

To systematically evaluate the numerical stability of the proposed algorithms, we employ a dual-scenario framework. All computations are performed in IEEE 754 double-precision floating-point arithmetic to ensure numerical fidelity.

\subsubsection{Markley Two-Vector Cases}

First, the algorithms were evaluated using the two-vector observation subsets from Markley's original 12 cases~\cite{markley1993_attitude}. The reference vectors $\qa[1], \qa[2]$, corresponding standard deviations $\sigma_1$, $\sigma_2$, and specific ground-truth orientation unit quaternions $\qq[T]$ were preserved exactly as defined by the author for each independent scenario. The vectors $\qb[\ell]$, $\ell= {1,2}$ are calculated as $\qb[\ell] = \qq[T][*] \qa[\ell] \qq[T]$. For each trial, we generate $M=10000$ noisy realizations by adding independent Gaussian noise with a standard deviation of $\sigma_1$, $\sigma_2$, respectively, to each component of $\qb[\ell]$.

\subsubsection{Monte Carlo Simulations}

Second, a Monte Carlo simulation was conducted over $N = 1000$ trials. For each trial we generate a uniformly distributed ground truth unit quaternion $\qq[T]$ and two pure quaternions $\qa[1], \qa[2]$ independently and uniformly distributed on the unit sphere. For each trial, we calculate two pure unit quaternions $\qb[1], \qb[2]$ as $\qb[\ell] = \qq[T][*]\qa[\ell]\qq[T]$, with $\ell = 1,2$. For each trial, we generate $M=1000$ noisy realizations by adding independent Gaussian noise with a standard deviation of $\sigma = 0.001$ to each component of $\qb[\ell]$, followed by normalization of $\qb[\ell]$ to unit length. 
 
\subsection{Numerical Stability and Topological Vulnerabilities}

The numerical validation metrics are evaluated via the mean accumulated Wahba's cost function value and the absolute angular error $\Delta\theta$. The absolute angular error $\Delta\theta$ is derived from the error quaternion $\Delta\qq$. Following the standard convention in attitude determination \cite{markley+crassidis2014_fundamenta}, $\Delta\qq$ represents the rotation from the ground truth $\qq[T]$ to the estimate $\qq$ given by $\Delta\qq = \qq\qq[T][-1] = \qq\qq[T][*]$. Given the axis-angle relationship $\Re(\Delta\qq) = \cos(\Delta\theta/2)$, the absolute error is computed as $\Delta\theta = 2\arccos\left|\Re(\qq\qq[T][*])\right|$, where the absolute value operator is employed since $\qq$ and $-\qq$ represent the same rotation. Lower values of both metrics indicate a closer agreement of the estimated quaternion $\qq$ with the ground truth $\qq[T]$.

Table~\ref{tab:markley_suite} reports the mean Wahba's cost over $M=10000$ noisy realizations of each of the Markley two-vector cases and Table~\ref{tab:MONTE_CARLO_SUMMARY} summarizes $\Delta\theta$ statistics and mean Wahba's cost from the Monte Carlo simulations.

\begin{table}[hbt!]
\centering
\caption{Markley's Two-Vector Cases: Mean Wahba's Cost}
\label{tab:markley_suite}
\small
\renewcommand{\arraystretch}{1.3}
\begin{tabularx}{\textwidth}{c Y Y Y Y}
\toprule
{Case} & {ESOQ2} & {Markley} & {FLAE} & {Proposed MARA} \\ 
\midrule
2  & $1.015431 \times 10^{-12}$ & $2.574101 \times 10^{-13}$ & $7.640959 \times 10^{-13}$ & $5.148202 \times 10^{-13}$ \\
4  & $2.407053 \times 10^{-01}$ & $2.436861 \times 10^{-05}$ & $7.474508 \times 10^{-05}$ & $4.873550 \times 10^{-05}$ \\
5  & $8.623568 \times 10^{-02}$ & $1.230889 \times 10^{-05}$ & $3.770580 \times 10^{-05}$ & $2.461731 \times 10^{-05}$ \\
7  & $9.999574 \times 10^{-05}$ & $5.177546 \times 10^{-03}$ & $4.818586 \times 10^{-03}$ & $6.254665 \times 10^{-10}$ \\
9  & $2.003890 \times 10^{-04}$ & $1.165998 \times 10^{-02}$ & $2.365439 \times 10^{-03}$ & $5.126744 \times 10^{-05}$ \\
11 & $1.243991 \times 10^{-05}$ & $3.999078 \times 10^{-02}$ & $9.065254 \times 10^{-01}$ & $2.487934 \times 10^{-05}$ \\
12 & $1.281932 \times 10^{-05}$ & $3.408773 \times 10^{-01}$ & $9.044711 \times 10^{-01}$ & $2.563815 \times 10^{-05}$ \\
\bottomrule
\end{tabularx}
\end{table}

Table~\ref{tab:markley_suite} reveals distinct algorithmic failure modes. In Case~2 all four algorithms attain costs of order $10^{-12}$, confirming numerical equivalence under well-conditioned geometry. Performance diverges markedly in ill-conditioned configurations. ESOQ2 records costs of $2.407 \times 10^{-1}$ and $8.624 \times 10^{-2}$ in Cases~4 and~5, respectively, attributable to sensitivity to near-collinear measurement vectors. In Cases~11 and~12, Markley's method yields costs of $3.999 \times 10^{-2}$ and $3.409 \times 10^{-1}$, while FLAE diverges catastrophically, with costs of $9.065 \times 10^{-1}$ and $9.045 \times 10^{-1}$ exceeding $0.90$. MARA does not exceed $5.127 \times 10^{-5}$ in any case, demonstrating uniform numerical stability across all seven configurations.

\begin{table}[hbt!]
\centering
\caption{Random Monte Carlo Simulations Statistical Summary}
\label{tab:MONTE_CARLO_SUMMARY}
\small
\renewcommand{\arraystretch}{1.3}
\begin{tabularx}{\textwidth}{l Y Y Y Y}
\toprule
& \multicolumn{3}{c}{{Absolute Angular Error $\Delta\theta$ (deg.)}} & {Wahba's Cost Function} \\ 
\cmidrule(lr){2-4} \cmidrule(lr){5-5}
{Algorithm} & {Mean} & {Std. Dev.} & {Max.} & {Mean} \\ 
\midrule
ESOQ2   & $0.291174$  & $5.569008$  & $176.094672$ & $1.565707 \times 10^{-03}$ \\
Markley & $55.035047$ & $38.137502$ & $179.999754$ & $3.828801 \times 10^{-01}$ \\
FLAE    & $100.998240$ & $47.602441$ & $179.999953$ & $6.033440 \times 10^{-01}$ \\
MARA    & $0.121737$  & $0.113256$  & $4.726589$   & $5.010550 \times 10^{-07}$ \\
\bottomrule
\end{tabularx}
\end{table}

Table~\ref{tab:MONTE_CARLO_SUMMARY} summarizes the results over random geometries. MARA attains a mean $\Delta\theta=0.121737 \text{ deg.}$, standard deviation $0.113256\text{ deg.}$, maximum $4.726589\text{ deg.}$, and mean Wahba cost $5.010550 \times 10^{-7}$. In addition, by choosing $\lambda_1 = 1$, $\mu_1 = 0$ for the calculation of $q_1$, that is using only the quaternion square root, yields identical numerical results, confirming the invariance of the solution in Theorem~\ref{theorem:SOLUTION_TO_THE_HOMOGENEOUS_WAHBAS_EQUATION} to the choice of the parameters $\lambda_1$ and $\mu_1$. ESOQ2 yields a comparable mean of $0.291174\text{ deg.}$, standard deviation $5.569008\text{ deg.}$, maximum $176.094672\text{ deg.}$, consistent with sporadic failure under degenerate measurement geometries. Markley's method and FLAE are globally unreliable, with mean errors of $55.035047\text{ deg.}$ and $100.998240\text{ deg.}$, and maxima approaching $180\text{ deg.}$, and mean Wahba costs of $3.828801 \times 10^{-1}$ and $6.033440 \times 10^{-1}$, respectively. Among the algorithms evaluated, MARA achieves the lowest mean of absolute angular error and Wahba's cost function.

The empirical distributions of absolute angular errors are visualized via boxplots in Figure~\ref{fig:boxplot_abs}. Each boxplot summarizes the error distribution: the box denotes the interquartile range (IQR), the red line indicates the median, and whiskers extend to $1.5 \times \text{IQR}$. Outliers are represented as individual points to illustrate the behaviour of the algorithms under extreme geometric configurations.
\begin{figure}[!t]
     \centering
         \includegraphics[width=0.5\linewidth]{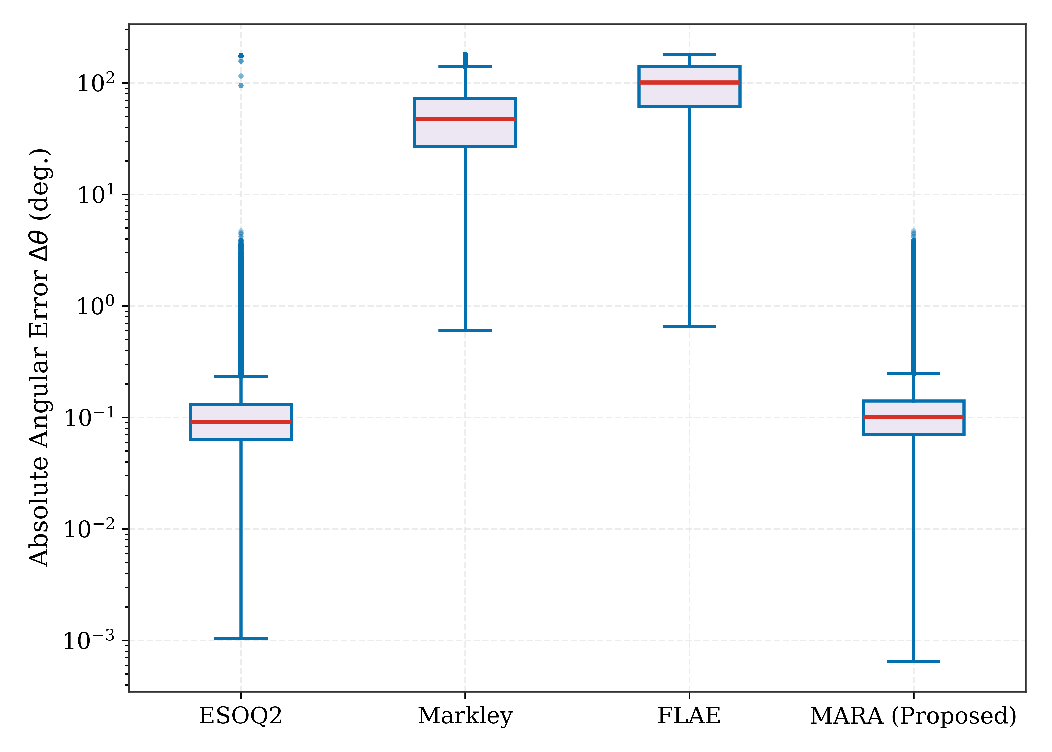}
         \if\ArxivFormat1
         	\caption{Absolute angular error $\Delta\theta$.}
         \else
         	\caption{Absolute angular error $\mathbf{\Delta\boldsymbol\theta}$.}
         \fi         
         \label{fig:boxplot_abs}
\end{figure}

Figure~\ref{fig:boxplot_abs} displays the log-scale distribution of absolute angular error $\Delta\theta$ for all four algorithms, in agreement with the Monte Carlo statistics of Table~\ref{tab:MONTE_CARLO_SUMMARY}. MARA exhibits the most compact distribution, with a median near $10^{-1}\text{ deg.}$, a narrow interquartile range, standard deviation $0.113256\text{ deg.}$, and maximum $4.726589\text{ deg.}$ ESOQ2 attains a similar median but displays a substantially heavier tail, with standard deviation $5.569008\text{ deg.}$ and outliers extending to $176.094672\text{ deg.}$, reflecting sporadic catastrophic failures. Markley's method and FLAE both show severely degraded distributions, with mean errors of the tested scenarios.

\section{Conclusion}\label{sec:CONCLUSION}
We have derived the necessary and sufficient conditions for the existence of a zero-cost solution to Wahba’s problem for the case of two vector observations. By establishing a direct algebraic connection between Wahba’s cost function and the homogeneous singular Sylvester equation, we provided a complete, closed-form parametrization of the solution set within the quaternion algebra. This approach avoids traditional matrix-based reformulations, offering a more transparent geometric interpretation of the problem's underlying structure. Based on this theoretical framework, we formulated the Minimal Analytical Rotation Algorithm (MARA) as an efficient, non-iterative solver for optimal attitude determination. Extensive Monte Carlo benchmarking, involving $10^6$ trials across a uniform distribution of geometric configurations, demonstrates that MARA achieves higher numerical stability than the established solvers such as ESOQ2. Crucially, the algorithm provides a significant reduction in computational overhead, specifically a $35.11\%$ reduction in total FLOPs, while maintaining a fixed execution time.

\section*{Acknowledgement}
The authors would like to thank the Federal Ministry of Research, Technology, and Space (BMFTR) for its support as part of the research program Communication Systems “Souverän. Digital. Vernetzt.”. Joint project 6G-life, project identification number: 16KIS2413K and the German Research Foundation (DFG, Deutsche Forschungsgemeinschaft) as part of Germany’s Excellence Strategy – EXC 2050/2 – Project ID 390696704 – Cluster of Excellence “Centre for Tactile Internet with Human-in-the-Loop” (CeTI) of TUD Dresden University of Technology.

\if\ArxivFormat1
  \bibliographystyle{unsrt}
  \bibliography{paper.bib}
\else
  \bibliography{paper.bib}
\fi

\end{document}